\documentclass[12pt]{amsart}
\usepackage{amstext,amsfonts,amssymb,amscd,amsbsy,amsmath,verbatim,fullpage}
\usepackage[alphabetic,lite,backrefs]{amsrefs} 
\usepackage[bookmarks, colorlinks=true, linkcolor=blue, citecolor=blue, urlcolor=blue, pagebackref=true]{hyperref}
\usepackage{ifthen}
\usepackage{color,tikz}
\usepackage{amsthm}
\usepackage{latexsym}
\usepackage[all]{xy}
\usepackage{enumerate}

\newtheorem{lemma}{Lemma}[section]

\newtheorem{cor}[lemma]{Corollary}

\newtheorem{claim*}{Claim}
\newtheorem{thm}[lemma]{Theorem}

\theoremstyle{remark}
\newtheorem{remark}[lemma]{Remark}
\newtheorem{example}[lemma]{Example}

\newcommand{\Lotimes}{\overset{L}{\otimes}}

\newcommand{\length}{\operatorname{length}}

\newcommand{\Spec}{\operatorname{Spec}}
\newcommand{\Ext}{\operatorname{Ext}}
\newcommand{\Tor}{\operatorname{Tor}}

\newcommand{\Proj}{\operatorname{Proj}}

\newcommand{\cO}{{\mathcal O}}

\newcommand{\codim}{\operatorname{codim}}

\newcommand{\CC}{\mathbb C}

\newcommand{\QQ}{\mathbb Q}

\newcommand{\ZZ}{\mathbb Z}
\newcommand{\HS}{\operatorname{HS}}

\title{Divergent series and Serre's intersection formula for graded rings}
\author{Daniel Erman}
\address{Department of Mathematics, University of Wisconsin, Madison, WI}
\email{\href{mailto:derman@math.wisc.edu}{derman@math.wisc.edu}}
\urladdr{\url{http://math.wisc.edu/~derman/}}
\begin{document} 
\thanks{The author was partially supported by NSF grant DMS-1302057.}
\maketitle


\vspace{-.25cm}
On a smooth variety, Serre's intersection formula~\cite[V.3]{serre} computes intersection multiplicities via an alternating sum of the lengths of $\Tor$ groups.  When the variety is singular, the corresponding sum can be a divergent series.  But there are alternate geometric approaches for assigning (often fractional) intersection multiplicities in some singular settings.  Our motivating question comes from Fulton, who asks whether an analytic continuation of the divergent series from Serre's formula can be related to these fractional multiplicities~\cite[Example 20.4.4]{fulton}.

In this note, we show how ideas from Avramov and Buchweitz~\cite{ab} answer Fulton's question in the context of graded rings.  Namely, \cite{ab} use a rational function to provide an analytic continuation of the Hilbert series of the total $\Tor$ module associated to an intersection (see Theorem \ref{thm:ab}).  Evaluating this rational function at $t=1$ provides a notion of intersection multiplicity that extends Serre's definition (see Corollary \ref{cor:main}).  We then show that the resulting multiplicity agrees with the multiplicities obtained for $\QQ$-Cartier divisors or on normal surfaces via alternate definitions used in algebraic geometry (see Theorem \ref{thm:QQCartier}).

For an isolated hypersurface singularity, similar techniques are used in~\cite[Theorem~2.2]{bvs} to study Hochster's Theta pairing.  Other related results are discussed in Remark~\ref{rmk}.

\bigskip

We illustrate this with an example.  Let $R:=\CC[x,y,z]/\langle f\rangle$ be the coordinate ring of a cubic curve.  Let $X=\Spec(R)\subseteq \mathbb A^3$.  If $L,L'\subseteq X$ are lines connecting distinct flex points to the origin, then $3L$ and $3L'$ are both Cartier divisors, and they intersect only at the origin $P\in X$.    Since $(3L\cdot 3L')_P=\deg X = 3$, the local intersection multiplicity $(L\cdot L')_P$ is $\frac{1}{3}$.

Now let $I,J$ be the ideals of these two lines.
Directly applying Serre's formula yields:
\[
\sum_{i\geq 0}^\infty (-1)^i\length\left( \Tor^R_i(R/I)\right)= 1-1+1-1+\dots
\]
To assign a value to this divergent series, we might first consider the series
\[
F(T):=\sum_{i=0}^\infty (-1)^i\length\left( \Tor^R_i(R/I,R/J)\right) \cdot T^i=1-T+T^2-T^3+\dots
\]
as suggested in~\cite[Example 20.4.4]{fulton}.  But the resulting analytic continuation at $T=1$ (by $\frac{1}{1+T}$) would have yielded a multiplicity of $F(1)=\frac{1}{2}$ instead of $\frac{1}{3}$.  

On the other hand, since everything is graded, there is another natural power series in the picture, 
namely 
the alternating sum of the Hilbert series of $\Tor_i(R/I,R/J)$. This is denoted $\chi(R/I,R/J)(t)$ in \cite{ab}. Since $\Tor_i$ is concentrated entirely in degree $\lfloor \frac{3i}{2}\rfloor$ we get
\[
\chi(R/I,R/J)(t) := \sum_{i=0}^\infty (-1)^i \HS_{\Tor_i(R/I,R/J)}(t)=1-t+t^3-t^4+t^6-t^7+\dots\in \ZZ[[t]]
\]
Using the rational function $\frac{1}{1+t+t^2}$ as our analytic continuation, we obtain the desired intersection multiplicity $\chi(R/I,R/J)(1)=\frac{1}{3}$.  

\bigskip

Generalizing this method provides an approach for defining intersection multiplicities in the singular case, at least as long as everything is graded.  
%
We fix a field $k$, let $R$ be a graded $k$-algebra (i.e. $R_0=k$), and let $M,N$ be finitely generated graded $R$-modules.  Following~\cite{ab}, we define
\begin{equation}\label{eqn:main def}\tag{*}
\chi(M,N)(t):=\sum_{i\geq 0}(-1)^i \HS_{\Tor_i(M, N)}(t),
\end{equation}
where $\HS_M(t)$ denotes the Hilbert series of a graded module $M$.  Note that $\chi(M,N)(t)$ can be thought of as the Hilbert series of the derived tensor product $M\Lotimes N$.

If either $M$ or $N$ has finite projective dimension, and if $M\otimes N$ has finite length, then $\chi(M,N)(t)$ is a Laurent polynomial in $\ZZ[t,t^{-1}]$.  Thus $\chi(M,N)(1)\in \ZZ$ and equals the alternating sum of the lengths of the $\Tor_i$, and it was shown by Peskine-Szpiro~\cite[Th\'{e}or\`{e}me, 2(iv)]{peskine-szpiro} that this extension of Serre's formula satisfies Vanishing and Positivity.\footnote{The corresponding result fails in the non-graded case.  See~\cite{dutta-hochster-mclaughlin}.}

In the case where $M$ and $N$ both have infinite projective dimension, $\Tor_i$ can be nonzero for infinitely many values of $i$, and thus Serre's formula would yield a divergent sum.  
However, Avramov and Buchweitz observed that $\chi(M,N)(t)$ admits a simple analytic continuation as a rational function.
\begin{thm}[Avramov-Buchweitz]\label{thm:ab}
Let $k$ be a field, $R$ a finitely graded $k$-algebra, and let $M,N$ finitely generated, graded $R$-modules.  Then
\[
\chi(M,N)(t) = \frac{1}{(1-t)^{\dim M + \dim N - \dim R}} \cdot e_{M,N}(t)
\]
where $e_{M,N}(t)\in \QQ(t)$ is a rational function.  Moreover, $e_{M,N}(1)\in \QQ_{>0}$.
\end{thm}
\noindent The theorem essentially appears as~\cite[Proposition~8]{ab}.\footnote{That proposition assumes that $R$ is generated in degree $1$, but their proof works equally well for a nonstandard grading.}  Using this rational function as an analytic continuation of $\chi(M,N)(t)$, we can always define $\chi(M,N)(1)\in \QQ\cup \{\infty\}$.  

\begin{cor}\label{cor:main}
Let $k$ be a field, $R$ a finitely generated $k$-algebra, and let $M,N$ finitely generated, graded $R$-modules.
\begin{enumerate}
	\item  $\chi(M,N)(1)=\infty \iff \dim M + \dim N > \dim R$.
	\item  Positivity: $\chi(M,N)(1)\in \mathbb Q_{>0} \iff \dim M+\dim N = \dim R$.
	\item Nonvanishing: $\chi(M,N)(1)=0 \iff \dim M + \dim N<\dim R$.
\end{enumerate}
\end{cor}
\noindent The corollary, which has a precursor in \cite[Th\'{e}or\`{e}me, 2(iv)]{peskine-szpiro} in the case where $M$ or $N$ has finite projective dimension, is an immediate consequence of Theorem~\ref{thm:ab}.

\bigskip

For a graded singularity $\Spec R$, one thus has an intersection multiplicity $\chi(M,N)(1)$.  The contribution of this note is the observation that this definition recovers the multiplicities obtained by two alternate approaches from algebraic geometry: from $\QQ$-Cartier divisors on normal varieties, or from a resolution of singularities of a normal surface.  

\begin{thm}\label{thm:QQCartier}
Let $R$ be a graded $k$-algebra  that is a normal domain, and let $X=\Spec R$.  Let $C,D\subseteq X$ be an integral curve and an effective divisor that intersect properly at the origin $P\in X$ and that are defined by homogeneous ideals $I_C,I_D\subseteq R$. If either: $D$ is $\QQ$-Cartier or $X$ is a surface, then we have
\[
(D\cdot C)_P = \chi(R/I_D,R/I_C)(1) \text{ and lies in } \mathbb Q_{>0},
\]
where $(D\cdot C)_P$ denotes the local intersection multiplicity at $P$.
\end{thm}
See~\cite[Example 7.1.16]{fulton} or ~\cite[II(b)]{mumford} for the definition of $(D\cdot C)_P$ on a normal surface.  For a Cartier divisor $E$ defined locally by a function $f$, and intersecting a curve $C$ properly at point $P$, we use the standard definition $(E\cdot C)_P:=\length \cO_{C,P}/\langle f\rangle$.  Then for a $\QQ$-Cartier divisor $D$ where $eD$ is Cartier, we take the natural local version of~\cite[Definition~1.4]{kollar}, setting $(D\cdot C)_P:-=\frac{1}{e}(eD\cdot C)_P$.

\bigskip

This note is structured as follows.  In \S\ref{sec:background}, we gather some of the background on Hilbert series that we will use in our proof.  \S\ref{sec:proof} contains the proof of Theorem~\ref{thm:QQCartier}. In \S\ref{sec:examples}, we consider examples of computing $\chi(M,N)(1)$ and we discuss the relation with some other results from the literature.

\section*{Acknowledgements}
We are grateful to many people for helpful conversations and references:
 Dima Arinkin, Lucho Avramov, Bhargav Bhatt, Ragnar-Olaf Buchweitz, Andrei C\u{a}ld\u{a}raru, William Fulton, Craig Huneke, and Kevin Tucker.

\section{Background}\label{sec:background}
In this section we summarize some results about Hilbert series.  Let $R$ be a finitely generated, graded $k$-algebra with generators in degrees $d_1,\dots,d_s$.  Let $S=k[x_1,\dots,x_s]$ with generators in the same degrees and with a surjection $S\to R$ of graded rings.  We write $\HS_M(t)$ for the Hilbert series of a graded module $M$.

\begin{lemma}\label{lem:hilb series}
For any finitely generated, graded $R$-module $M$, we have
\[
\HS_M(t)= \frac{1}{(1-t)^{\dim M}}\cdot e_M(t) \text{ where }  e_M(t)\in \QQ(t) \text{ and } e_M(1)\in \mathbb Q_{>0}.
\]
\end{lemma}
\begin{proof}See~\cite[Proposition 4.4.1]{bruns-herzog}, which is equivalent but in a slightly different form.
\end{proof}

\begin{lemma}\label{lem:rational hilb series} Let $M, N$ be finitely generated $R$-modules.  Then
\[
\chi(M,N)(t)=\frac{\HS_M(t)\HS_N(t)}{\HS_R(t)}.
\]
In particular, $\chi(M,N)(t)$ is a rational function.
\end{lemma}
\begin{proof}
This is proven in ~\cite[Lemma 7(ii)]{ab}, where it is cited as mathematical folklore.
\end{proof}

\begin{proof}[Proof of Theorem~\ref{thm:ab}]
By Lemma~\ref{lem:rational hilb series}, we have $\chi(M,N)(t) = \frac{\HS_M(t)\HS_N(t)}{\HS_R(t)}$.  We then apply Lemma~\ref{lem:hilb series} to obtain
\[
\chi(M,N)(t) =\frac{\left(\frac{1}{(1-t)^{\dim M}}\cdot e_M(t) \right)\cdot\left( \frac{1}{(1-t)^{\dim N}}\cdot e_N(t)\right)}{\frac{1}{(1-t)^{\dim R}}\cdot e_R(t)}=\frac{(1-t)^{\dim R}}{(1-t)^{\dim M+\dim N}} \cdot \frac{e_M(t)e_N(t)}{e_R(t)}.
\]
This is exactly the statement of~\cite[Proposition~8(i)]{ab}, which is only stated in the case where $R$ is generated in degree $1$, but their proof applies just as well in this context.  Since $e_M(1),e_N(1),e_R(1)$ are all strictly positive, it follows that $\frac{e_M(1)e_N(1)}{e_R(1)}\in \mathbb Q_{>0}$.  Setting $e_{M,N}(t):=\frac{e_M(t)e_N(t)}{e_R(t)}$ and evaluating at $t=1$, we obtain all three parts of the theorem.
\end{proof}


%
%

\section{Proof of Main Result}\label{sec:proof}

\begin{proof}[Proof of Theorem~\ref{thm:QQCartier}]  {\bf $\QQ$-Cartier case:}  
Since $D$ is $\QQ$-Cartier, we may assume that $eD$ equals a Cartier divisor $E$ defined by the homogeneous principal ideal $I_E=\langle f\rangle$.  By definition, we then have
\[
(D\cdot C)_P:=\tfrac{1}{e}(E\cdot C)_P=\tfrac{1}{e} \dim_k \cO_{C,P}/\langle f|_{C}\rangle.
\]
Since $C$ is not contained in $D$, it is also not contained in $E$ and thus $f|_{C}\ne 0$.  It follows that $\Tor^R_i(R_E,R_C)\ne 0$ if and only if $i=0$.  We can then compute $\chi(R_E,R_C)(1) = \dim_k \Tor_0(R_E, R_C)=\dim_k \cO_{C,P}/\langle f|_{C}\rangle$.  We thus have
\begin{equation}\label{eqn:DC chi}
(D\cdot C)_P = \tfrac{1}{e} \chi(R_E,R_C)(1).
\end{equation}

We next relate the constants  $m_D(1)$ and $m_E(1)$.  We define $X_0:=\Proj R$, which is a normal projective variety, and we let $D_0=\Proj(R/I_D)$ and $E_0=\Proj(R/I_E)$ be the corresponding divisors on $X_0$.  Note that $eD_0$ is still equivalent to the Cartier divisor $E_0$.  Moreover, if we write $H_0$ for the Weil divisor corresponding to $\cO_{X_0}(1)$, then $D_0, E_0,$ and $H_0$ are all positive rational multiples of one another, and hence each is ample.  

Let $n$ be the dimension of $X_0$.  For any $r$ we have the short exact sequence
\[
0\to \cO_{X_0}(-D_0+rH_0) \to \cO_{X_0}(rH_0)\to \cO_{D_0}(rH_0|_{D_0})\to 0.
\]
Thus, for $r\gg 0$ we compute
\begin{align*}
\dim (R/I_D)_r &= h^0(D_0,\cO_{D_0}(rH_0|_{D_0})) \\
&= h^0(X_0,\cO_{X_0}(rH_0)) - h^0(X_0,\cO_{X_0}(-D_0+rH_0)),
\intertext{Since $X_0$ is normal, we may apply Asymptotic Riemann-Roch to compute the above two terms up to $O(r^{n-2})$, and after some simple cancellations we obtain:}
&= \frac{H_0^{n-1} \cdot D_0}{(n-1)!}r^{n-2}+ O(r^{n-2}).
\end{align*}
In the above line, the intersection product is computed on $X_0$.  Since Lemma~\ref{lem:hilb series} implies that $\HS_{R/I_D}(t)=\frac{1}{(1-t)^{n}} m_D(t)$, we can also write $\dim (R/I_D)_r = \frac{m_D(1)}{(n-1)!} r^{n-1}+O(r^{n-2})$.  We thus conclude that 
\[
m_D(1) = H_0^{n-1} \cdot D_0.
\]
A similar computation yields $m_E(1)=H_0^{n-1}\cdot E_0$, which then implies that
\begin{equation}\label{eqn:me and md}
m_E(1)=e\cdot m_D(1).
\end{equation}

Finally, we turn our attention to comparing the series $\chi(R_E,R_C)(t)$ and $\chi(R_D,R_C)(t)$:
\begin{align*}
\chi(R_E,R_C)(t) - e\cdot \chi(R_D,R_C)(t)&= \frac{m_E(t)m_C(t)}{m_R(t)} - \frac{e\cdot m_D(t)m_C(t)}{m_R(t)}\\
&=(m_E(t)-e\cdot m_D(t)) \frac{m_C(t)}{m_R(t)}.
\intertext{But by \eqref{eqn:me and md}, $m_E(1)=e\cdot m_D(1)$, and so plugging in $t=1$ yields}
\chi(R_E,R_C)(1) - e\cdot\chi(R_D,R_C)(1)&=0.
\end{align*}
Combining this with \eqref{eqn:DC chi}, we conclude that
\[
(D\cdot C)_P= \tfrac{1}{e}\chi(R_E,R_C)(1) = \chi(R_D,R_C)(1).
\]

\smallskip

{\bf Normal surface case:}  Next we take the case where $X$ is a normal surface.  
We define $X_0=\Proj(R)$, which is a smooth curve since $R$ is a normal, graded ring of dimension $2$.  Let $H_0$ be the ample Cartier divisor corresponding to $\cO_{X_0}(1)$.  We define $C_0$ and $D_0$ as the divisors corresponding to $C$ and $D$ on $X_0$.  Since $C,D$ are integral, we have that $C_0$ and $D_0$ is each a single point on $X_0$.  For all $r\gg 0$, we have
\[
\dim(R/I_C)_r = H^0(C_0,\cO_{C_0}(r)) = \deg C_0=1.
\]
Thus $m_C(1)=1$.  Similarly $m_D(1)=\deg D_0=1$.  Moreover, for $r\gg 0$, we can use Riemann-Roch to compute
\[
\dim R_r = H^0(X_0,\cO_{X_0}(rH_0))=(\deg H_0)\cdot r + O(1)
\]
and hence we also have $m_R(1) = \deg H_0$.

Blowing up the cone point of $X$, we get a resolution $\pi: \widetilde{X}\to X$ with exceptional divisor $E$ where $E^2=-\deg H_0=-m_R(1)$.  
We use Mumford's definition (see~\cite[Example~7.1.16]{fulton} or \cite[II(b)]{mumford}) to define
\[
(C\cdot D)_P:=\pi^*_{\text{num}}C \cdot \pi^*_{\text{num}}D.
\]
where $\pi^*_{\text{num}}C$ is the unique $\QQ$-divisor on $\widetilde{X}$ that pushes forward to $C$ and that intersects the exceptional divisor in multiplicity $0$, and where $\pi^*_{\text{num}}D$ is similarly defined.  It follows that $\pi^*_{\text{num}}C=\widetilde{C}+\frac{1}{m_R(1)} E$ and similarly for $D$.  We then compute:
\begin{align*}
\pi^*_{\text{num}}C \cdot \pi^*_{\text{num}}D &= \widetilde{C}\cdot \widetilde{D} +\tfrac{1}{m_R(1)}\left(\widetilde{C}\cdot E+E\cdot \widetilde{D}\right) +\tfrac{1}{m_R(1)^2} \left(E^2\right) \\
&= 0+\tfrac{2}{m_R(1)}-\tfrac{1}{m_R(1)}  \\
&= \tfrac{1}{m_R(1)}\\
&=\tfrac{m_C(1)m_D(1)}{m_R(1)}.
\end{align*}
This is precisely equal to $\chi(R/I_D,R/I_C)(1)$ by the proof of Theorem~\ref{thm:ab} from above.
\end{proof}

\section{Examples and remarks}\label{sec:examples}
%

\begin{example}[Rational normal cone]
Let $X\subseteq \mathbb A^{d+1}=\Spec(k[x_0,\dots,x_d])$ be the cone over the rational normal curve of degree $d$.   Let $L$ be the class of a line through the origin.  Note that $L$ is a $\QQ$-Cartier divisor and $L^2=\frac{1}{d}$.
%
Let $R$ be the coordinate ring of $X$, so that $R$ is $k[x_0,\dots,x_d]$ modulo the $2\times 2$ minors of the matrix $\begin{pmatrix} x_0&x_1&\dots&x_{d-1}\\x_1&x_2&\dots&x_d\end{pmatrix}$.  Let $I_1=\langle x_0,\dots,x_{d-1}\rangle$ and $I_2=\langle x_1,\dots,x_d\rangle$ be ideals of two different lines $L_1$ and $L_2$ on $X$.  

A direct computation (using e.g. Lemma~\ref{lem:rational hilb series}) yields:
\begin{align*}
\chi(R/I_1,R/I_2)(t) &= 1-(d-1)t+(d-1)^2t^2-(d-1)^3t^3+(d-1)^4t^4-\dots \\
&= \frac{1}{1+(d-1)t}.\\
\end{align*}
Evaluating at $t=1$ we get $\chi(R/I_1,R/I_2)(1)= \frac{1}{1+(d-1)}=\frac{1}{d}$.
%
%
\end{example}
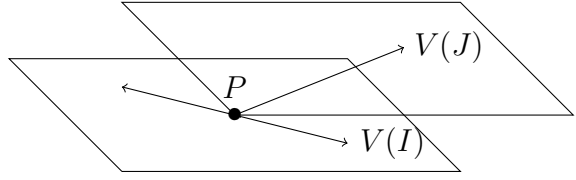
\begin{figure}
\begin{tikzpicture}[scale=1.5]
\draw[-](4,4)--(7,4)--(6,5)--(3,5)--(4,4);
\draw[-](3,3.5)--(6,3.5)--(5,4.5)--(2,4.5)--(3,3.5);
\draw[<->] (4,4)--(5.5,4.6);
\draw[<->] (3,4.25)--(5,3.75);
\draw(5.4,3.75) node {$V(I)$};
\draw(5.9,4.6) node {$V(J)$};
\draw(4,4) node {$\bullet$};
\draw(4,4.25) node {$P$};
\end{tikzpicture}
\caption{In Example~\ref{ex:non normal}, $X$ is the union of the two planes meeting at a point $P$.  Two lines $V(I)$ and $V(J)$ intersect at $P$ with multiplicity $\frac{1}{2}$.}
\label{fig:bracket}
\end{figure}
\begin{example} \label{ex:non normal}The intersection formulas from Corollary~\ref{cor:main} make sense even when $X$ fails to be normal.
For instance, let $R = \QQ[x,y,z,w]/(xz,xw,yz,yw)$ be the coordinate ring of two planes in $\mathbb A^4$ intersecting at the origin.  Let $I = (x,y,w)$  and $J=(y,z,w)$ be the ideals of two lines.  Then
\begin{align*}
\chi(R/I,R/J)(t)&= (1-2t+5t^2-12t^3+29t^4-70t^5+169t^6-\dots) \\
&= \left(\frac{1}{1+2t-t^2}\right).
\end{align*}
Hence $\chi(R/I,R/J)(1)=\frac{1}{2}$.  
\end{example}

\begin{example}\label{ex:cone over quadric}
Let $X=\Spec(k[x,y,z,w]/(xw-yz))$ be the cone over the quadric surface.  Let $D=V(x,y)$ and $E=V(z,w)$, which are both planes lying in $X$.   Note that $D\cap E$ is supported on the origin, so we have
\[
3=\codim D\cap E>\codim D +\codim E=2.
\]

Write $I_D, I_E$ for the defining ideals of $D$ and $E$.  We have
\begin{align*}
\chi(R/I_D, R/I_E)(t) &= 1 + t^2+t^4+\dots \\
&= \frac{1}{1-t^2}.
\end{align*}
This has a pole at $t=1$, hence $\chi(R/I_D,R/I_E)(1)=\infty$.  
\end{example}

\begin{remark}\label{rmk}
We conclude with several remarks:
\begin{enumerate}
	\item  A similar rational function expression for the total $\Ext$-module $\Ext^*(M,N)$ is given in~\cite{abs} where it was used to recover and greatly expand on a ramification formula for modular invariants from~\cite{bens-craw}.
	\item  In~\cite{bvs}, the authors use the series $\chi(M,N)(t)$ in their analysis of Hochster's Theta pairing in the case of a graded hypersurface singularity.  The authors go on to analyze the nongraded case as well, relating the Theta pairing to a linking number of associated cycles, thus proving a conjecture of Steenbrink.
	\item  In~\cite{gulliksen}, Gulliksen considers another way to obtain a Serre-like intersection formula in the case of non-regular rings (and not just in the graded case).  In particular, he proves that if you take a surjective map $S\to R$, where $S$ is a regular ring of minimal embedding dimension, then the resulting sum
	\[
	\chi^S(M,N) := \sum_{i} (-1)^i \dim \Tor^S_i(M,N)
	\]
	will not depend on the choice of the ring $S$.  In the graded case, this formula gives different results than the formula $\chi(M,N)(1)$.  For instance, in the example from the introduction, we could choose $S:=\CC[x,y,z]$ and then
	\[
	\chi^S(R/I,R/J) =  3 - 3 = 0.
	\]
In fact, this equals zero by~\cite[IV.3]{serre} because $\dim R/I + \dim R/J=\dim R<\dim S$, and something similar will hold for every non-trivial example where Corollary~\ref{cor:main} yields a positive, rational multiplicity.
	\item  Serre's definition contains the implicit idea of attaching a multiplicity to an object like $R/I\Lotimes_R R/J$.  A similar idea appears in~\cite{gorbounov-schechtman}, where the authors use techniques for summing divergent series to show how to recover numerical data, like degree, in terms of homological data contained in what they call a semi-free resolution.
	\item For any flat family of graded modules $M_t$ over a base $T$, the Hilbert series $M_t\Lotimes N$ does not depend on $t\in T$.  In particular, $\chi(M,N)(1)$ is finite whenever $\dim M+\dim N\leq \dim R$, regardless of whether $M\otimes N$ has finite length.  It is thus possible to define degenerate intersection multiplicities like $\chi(R/I,R/I)(1)$.
	\item  The proof of Theorem~\ref{thm:QQCartier} relies on the fact that $m_{R/I_D^e}(1)=e \cdot m_{R/I_D}(1)$ when $R$ is normal and $I_D$ is a cxdimension $1$ ideal.  This can fail when $R$ is not normal.  For instance, let $R=\Spec(\CC[x,y,z]/\langle y^2z-x^3\rangle)$ and let $D$ be the line defined by the ideal $I_D=\langle x,y\rangle$.  Then $m_{R/I_D}(1)=1$ whereas $m_{R/I_D^2}(1)=3$.
\end{enumerate}
\end{remark}

\end{document}